\documentclass[11pt]{amsart}

\usepackage{amsmath,amssymb,amsthm}

\usepackage{amsmath, amssymb}
\usepackage{amsthm, amsfonts, mathrsfs}
\usepackage{fullpage}
\usepackage{amsfonts,graphicx}

\usepackage{amsmath,amssymb,graphicx,epsfig,epic,mathrsfs}

\newcommand{\DD}{\textnormal{D}}

\newtheorem{thm}{Theorem}
\newtheorem{lem}{Lemma}
\newtheorem{cor}{Corollary}
\newtheorem{prop}{Proposition}

\theoremstyle{definition}
\newtheorem{defi}{Definition}

\theoremstyle{remark}
\newtheorem{Rema}{Remark}
\newtheorem*{rema*}{Remark}

\newcommand{\NN}{\mathbb{N}}
\newcommand{\RR}{\mathbb{R}}
\newcommand{\ZZ}{\mathbb{Z}}

\DeclareMathOperator{\divergence}{div}
\newcommand{\restr}[2]{\ensuremath{\left. #1 \right|_{#2}}}

\newcommand{\cepsilon}{c _\varepsilon}

\newcommand{\cepsiloninit}{c _{0, \varepsilon} }

\newcommand{\vepsilon}{v _\varepsilon }

\newcommand{\vepsiloninit}{v _{0, \varepsilon} }

\author[T. Hmidi]{Taoufik Hmidi}
\address{IRMAR, Universit\'e de Rennes 1\\ Campus de
Beaulieu\\ 35~042 Rennes cedex\\ France}
\email{thmidi@univ-rennes1.fr}
\author[S. Sulaiman]{Samira Sulaiman}
\address{IRMAR, Universit\'e de Rennes 1\\ Campus de
Beaulieu\\ 35~042 Rennes cedex\\ France}
\email{samira.sulaiman@univ-rennes1.fr}
\date{}

\begin{document}
\title
[About isentropic Euler system]
{Incompressible limit for the 2D isentropic Euler system with critical initial data}

\date{\today}

\maketitle

\begin{abstract} We study in this paper the low Mach number limit for the $2d$ isentropic Euler system with ill-prepared initial data belonging to the critical Besov space $B_{2,1}^2$. By combining Strichartz estimates with the special structure of the vorticity we prove that  the lifespan  goes to infinity as the Mach number goes to zero. We also prove the strong convergence in the space of the initial data $B_{2,1}^2$  of the incompressible parts  to the solution of the  incompressible Euler system. There are at least two main difficulties: the first one concerns the Beale-Kato-Majda criterion which is not known to work for rough regularities. However, the second one is related to the critical aspect of  the Strichartz norm   $\Vert(\textnormal{div}\,v_{\varepsilon},\nabla c_{\varepsilon})\Vert_{L_{T}^{4}L^\infty}$ which has the scale of $B_{2,1}^2$ in the space variable.
\end{abstract}

\tableofcontents

\section{Introduction}

This work is devoted to the study  of   the  isentropic  Euler equations describing the evolution of a compressible fluid evolving in the whole space and under a constant entropy constraint. The system is given by the coupled equations
\begin{equation}\label{Euc} 
\left\{ \begin{array}{l} 
\rho(\partial_t u+u\cdot\nabla u)+\nabla \rho^\gamma=0, \, \gamma>1\\ 
\partial_{t}\rho+\textnormal{div}\,(\rho \, u)=0\\
(u,\rho)_{|t=0}=(u_{0},\rho_{0}).\\  
\end{array} \right.
\end{equation}
Here, $u$ stands for the velocity field which is a time-dependent vector field over $\RR^d$, $\rho>0 $  is the density which is assumed to be nonnegative in order to avoid the vacuum. The parameter   $\gamma>1$ is the adiabatic exponent. 
The local well-posedness theory can be achieved  by using   the symmetrization  of Kawashima Makino and Ukai \cite{MKaw}. In other words,   we introduce  the sound speed defined by,
$${c}=2\frac{\sqrt\gamma}{\gamma-1}\rho^{\frac{\gamma-1}{2}}$$
 and thus the system \eqref{Euc} reduces to 
\begin{equation}\label{l01} 
\left\{ \begin{array}{l} 
\partial_{t}u+u\cdot\nabla u+\bar{\gamma}\, c\,\nabla\, c=0\\ 
\partial_{t}c+u\cdot\nabla c+ \bar\gamma\,c\,\textnormal{div}\,u=0\\
(u,c)_{|t=0}=(u_0,c_0),\\  
\end{array} \right.
\end{equation}
with  $\bar\gamma:=\frac{\gamma-1}{2}.$ This form fits with the theory of the symmetric quasilinear hyperbolic systems developed by  Klainerman and Majda and which answers   the local well-posedness in the framework of Sobolev spaces $H^s, s>\frac{d}{2}+1,$ \mbox{see \cite{km81,km82}.} It is well-known that the system \eqref{l01} can be considered as   a good approximation of the incompressible Euler equations, despite we make some suitable assumptions as we will see later. The first step to reach this result is to introduce  a suitable rescaling and change of variables. This can be done  by making a perturbation of order $\varepsilon$ from the constant background state \,$(0,c_0)$:
$$
u(t,x)=\bar\gamma c_0\varepsilon v_\varepsilon(\varepsilon\bar\gamma c_0 t,x),\, c(t,x)=c_0+\bar\gamma c_0\varepsilon c_\varepsilon(\varepsilon\bar\gamma c_0 t,x).
$$
Therefore we get the system
\begin{equation}\label{C1} 
\left\{ \begin{array}{l} 
\partial_{t}v_{\varepsilon}+v_{\varepsilon}\cdot\nabla v_{\varepsilon}+\bar\gamma c_{\varepsilon}\nabla c_{\varepsilon}+\frac{1}{\varepsilon}\nabla c_{\varepsilon}=0\\ 
\partial_{t}c_{\varepsilon}+v_{\varepsilon}\cdot\nabla c_{\varepsilon} +\bar\gamma c_{\varepsilon}\textnormal{div}\,v_{\varepsilon}+\frac{1}{\varepsilon}\textnormal{div}\,v_{\varepsilon}=0\\
(v_{\varepsilon},c_{\varepsilon})_{| t=0}=(v_{0,\varepsilon},c_{0,\varepsilon}).\\  
\end{array} \right.
\end{equation}
The parameter $\varepsilon$ is called the Mach number and measures the compressibility of the fluid. For more details about the derivation of this last model we refer for instance  to 
\cite{dh04,km81,GS01}. The mathematical study of this model is well developed and  there are a lot of papers  devoted   to the rigorous justification of the convergence  towards the incompressible Euler equations when the Mach number goes to zero. We recall that the incompressible Euler  system is given by,
\begin{equation}
   \label{c2}
   \left\{
     \begin{array}{l}
       \partial _t v + v \cdot \nabla v + \nabla p=0 \\
       \divergence v =0 \\
       \restr{ v }{ t=0 } = v _0 .
     \end{array}
   \right.
\end{equation}
 We point out that the issue to this problem depends on several  factors like the geometry of the domain    where the fluid is assumed to evolve or the state of the initial data.  For example, the analysis in the case of  the full space is completely different from  the torus due to   the resonance phenomenon.  The second factor is related to the structure of  the initial data and  whether they are well-prepared or not. In the well-prepared case, the initial data are assumed to be slightly compressible, which means that $(\divergence \vepsiloninit,\nabla \cepsiloninit) =O(\varepsilon)$ as $\varepsilon \rightarrow 0 $. This assumption allows to get a uniform bound \mbox{for $(\partial_t\vepsilon)_{\varepsilon},$} and consequently we can  pass to the limit by using Aubin-Lions compactness lemma, for more details  see  \cite{km81,km82}. However, in
the ill-prepared case, the family $(\vepsiloninit,\cepsiloninit)_{\varepsilon} $ 
 is only assumed to be bounded in some Sobolev spaces $H^s$ with $ s>\frac{d}{2}+1$ and  the incompressible parts of $(\vepsiloninit)_{\varepsilon} $ converge to some vector field $ v _0 $.   This case  is more subtle because the time derivative  $\partial_t\vepsilon,$ is of size $O(\frac{1}{\varepsilon})$ and the preceding compactness method is out of use. To overcome this difficulty, Ukai \cite{u86} used  the  dispersive effects generated by the acoustic waves in order to prove that the compressible part of the velocity and the acoustic term vanish when $\varepsilon$ goes to zero.   
We  point out that this problem has already been studied in numerous
papers, see for instance 
\cite{aaz,A87,d,D02,D99,km81,km82,L95,LM98,M84,GS01, S87,u86}.

Concerning the blow up theory, it is well-known that contrary to the incompressible Euler system, the equations \eqref{Euc} develop in space dimension two singularities in finite time and for some smooth initial data, see \cite{r89}. This phenomenon holds true for higher dimensions, see \cite{st85}. On the other hand, it was shown in  \cite{km82} that for the whole space or for the torus domain when the limit system \eqref{c2} exists for some \mbox{time interval $[0,T_0]$} then for ever $T<T_0$ there exists $\varepsilon_0>0$ such that  for $\varepsilon<\varepsilon_0$ the solution to \eqref{C1} exists till the time $T$.  This result was established for the well-prepared case and with sufficient smooth initial data, that is, in  $H^s$ and $s>\frac{d}{2}+2.$ The proof relies on the perturbation theory which ceases to work for lower regularities $\frac{d}{2}+2>s>\frac{d}{2}+1$.   Consequently, if we denote by $T_\varepsilon$ the  lifespan of the solution $(\vepsilon,\cepsilon)$ then we get  in space dimension two that  $\lim_{\varepsilon\to0}T_\varepsilon=+\infty.$  It seems that we can get better information about $T_\varepsilon$  when the initial data have some special structures. In \cite{A93}, Alinhac proved that in space dimension two and for axisymmetric initial data the lifespan of the solution is equivalent to $\frac1\varepsilon$. In dimension three and for irrotational velocity, Sideris \cite{st85} established that  the solutions are almost global in time, that is, their lifespan   are bounded below by $e^{\frac{1}{\varepsilon}} $. Finally, we mention the results of  \cite{Grassin,Serre} dealing with the global existence under some suitable conditions on the initial data: the initial density must be small and has a compact support and the spectrum of $\nabla u_0$ must be far away the set of the negative real numbers.

We intend in this paper to deal with the incompressible limit problem  in space dimension two but for the critical  Besov regularity $B_{2,1}^2$. We emphasize that by using Strichartz estimates  we can prove that for initial data  with sub-critical regularity, that is  $H^s$ and $ s>2$, the lifespan is bounded below by $C\log\log\frac1\varepsilon.$ We can interpret this double logarithmic growth by the fact that the Sobolev norm of the solutions to the incompressible Euler equations has at most a double exponential growth in time. 
Our main result reads as follows.

\begin{thm}\label{theo2}
Let $\{(v_{0,\varepsilon},c_{0,\varepsilon})_{0<\varepsilon\le 1}\}$ be a family of  initial data such that 
$$\sum_{q\ge-1}2^{2q}\sup_{0<\varepsilon\le 1}\Vert(\Delta_{q}v_{0,\varepsilon},\Delta_{q}c_{0,\varepsilon})\Vert_{L^{2}}<+\infty.$$
Then the system \eqref{C1} has a unique solution $(v_{\varepsilon}, c_{\varepsilon})\in\mathcal{C}([0,T_{\varepsilon}]; B_{2,1}^{2})$ with
$$\lim_{\varepsilon\to0}T_{\varepsilon}=+\infty.$$
Moreover, the acoustic parts of the solutions tend to zero: 
$$\lim_{\varepsilon\to 0}\Vert(\textnormal{div}\,v_{\varepsilon},\nabla c_{\varepsilon})\Vert_{L_{T_{\varepsilon}}^{1}L^\infty}=0.$$
Assume in addition that  $\displaystyle{\lim_{\varepsilon\to0}\|\mathbb{P}v_{0,\varepsilon}-v_0\|_{L^2}=0},$ for some  $v_0\in B_{2,1}^2$. Then the incompressible \mbox{parts $(\mathbb{P} v_{\varepsilon})_{\varepsilon}$} converge to the solution $v$ of the system \eqref{c2} associated to the initial data $v_0.$ More precisely, for every $T>0$,
$$\lim_{\varepsilon\to0}\|\mathbb{P}v_{\varepsilon} -v\|_{L^\infty_TB_{2,1}^2}=0.$$
Where we denote by $\mathbb{P}v=v-\nabla\Delta^{-1}\textnormal{div } v$, the Leray's projector over solenoidal vector fields.
\end{thm} 
Before discussing the important ingredients of the proof, we will give some useful remarks.
\begin{Rema}\label{rimk1}
\begin{enumerate} 
\item The result of   Theorem \ref{theo2} will be generalized in Theorem \ref{ttheo2}. This latter one gives a precise information about the lifespan: we get 
$$T_\varepsilon\geq C_0\log\log\Psi(\log\frac1\varepsilon),$$
where $\Psi$ is an implicit function depending on the profile of the initial data.
\item In the second part of Theorem \ref{theo2}, the velocity $v_0$ belongs naturally to the space $L^2$, but according to the assumption $\sum_{q\ge-1}2^{2q}\sup_{0<\varepsilon\le 1}\Vert\Delta_{q}v_{0,\varepsilon}\Vert_{L^{2}}<+\infty$, combined with the weak compactness of the Besov space $B_{2,1}^2$, we obtain $v_0\in B_{2,1}^2.$ Moreover we get the strong convergence to $v_0$  in the space $B_{2,1}^2$. Indeed, let $N\in \NN^*$, then we can write
\begin{eqnarray*}
\|\mathbb{P}v_{0,\varepsilon}-v_0\|_{B_{2,1}^2}\lesssim 2^{2N}\|\mathbb{P}v_{0,\varepsilon}-v_0\|_{L^2}+\sum_{q\geq N} 2^{2q}\big(\|\Delta_q v_0\|_{L^2}+\sup_{0<\varepsilon\le 1}\Vert\Delta_{q}v_{0,\varepsilon}\Vert_{L^{2}}\big).
\end{eqnarray*}
Thus, letting $\varepsilon$ go to zero, we infer
$$
\limsup_{\varepsilon\to 0}\|\mathbb{P}v_{0,\varepsilon}-v_0\|_{B_{2,1}^2}\lesssim \sum_{q\geq N} 2^{2q}\big(\|\Delta_q v_0\|_{L^2}+\sup_{0<\varepsilon\le 1}\Vert\Delta_{q}v_{0,\varepsilon}\Vert_{L^{2}}\big)
$$
and therefore when $N$ goes to $+\infty$ yields the desired result.
\end{enumerate}
\end{Rema}
Let us now describe the main   difficulties that emerge  when we try to work in the critical framework and  give the basic ideas to solve them. We distinguish   two principal difficulties: The first one has a compressible nature and concerns the estimate of the acoustic parts $\Vert(\textnormal{div}\,v_{\varepsilon},\nabla c_{\varepsilon})\Vert_{L_{T}^{1}L^\infty}$. Since in the space variable the norm of the acoustic parts scales like the space of the initial data, it seems then so hard to get an explicit rate on $\varepsilon$ for these quantities. It is worthy pointing out that for the sub-critical regularities the rate convergence to zero is of order $\varepsilon^\theta,$ for some $\theta>0.$  To overcome this critical aspect, we will proceed in the spirit of the the paper \cite{Hmidi-comp}: we can find an adequate nondecreasing  profile $\Psi$ such that $\lim_{q\to+\infty}\Psi(q)=+\infty$ and:
$$\sup_{\varepsilon}\sum_{q}\Psi(q)2^{2q}\|(\Delta_qv_{0,\varepsilon},\Delta_qc_{0,\varepsilon})\|_{L^2}<+\infty.$$
Then by performing  energy estimate we prove  that the solutions have the same additional decay in the time interval $[0,T_\varepsilon]$. This means that there is no energy transfer to higher frequencies. Consequently,   interpolating between this estimate and Strichartz estimate for the lower degree quantity $\|(\nabla\Delta^{-1}v_\varepsilon, c_\varepsilon)\|_{L^1_TL^\infty},$ we get that the acoustic parts vanish when the Mach number  goes to zero. However the rate convergence is implicit and can be related to the distribution on Fourier modes of the energy of the initial data.

The second difficulty that one should to deal with  has an incompressible nature and concerns the  Beale-Kato-Majda criterion \cite{BKM}. It was initially stated for the system \eqref{c2} and for the  sub-critical spaces $H^s, s>\frac{d}{2}+1$.  It asserts that  we can propagate the initial regularity in the time interval $[0,T]$ provided that we can control the quantity $\|\omega\|_{L^1_TL^\infty}$, where $\omega$ denotes the vorticity of the velocity. As a matter of fact, we get in space  dimension two the global existence of the strong solutions since the vorticity is  only transported by the flow and thus its $L^\infty$ norm is conserved. Unfortunately, this criterion is not known to work in the critical Besov space $B_{2,1}^{\frac{d}{2}+1}$  and it is replaced by a stronger norm  $\Vert\omega\Vert_{L^1_TB_{\infty,1}^{0}}.$ The global existence for $2d$ case was solved some years ago by Vishik \cite{vis} who was able to prove the following  linear growth:
$$\|\omega(t)\|_{B_{\infty,1}^0}\le C\|\omega_0\|_{B_{\infty,1}^0}\Big(1+\int_0^t\|\nabla v(\tau)\|_{L^\infty}d\tau\Big).$$
The Vishik's method relies on a  composition law in Besov space $B_{\infty,1}^0$ with a logarithmic growth.  We emphasize that in \cite{hmid-kera}, we gave an other proof which has the advantage to fit with  general models like transport-diffusion equations. Now, in order to get a lower bound for the lifespan satisfying $\displaystyle{\lim_{\varepsilon\to 0}T_\varepsilon=+\infty}$  we should establish an analogous estimate in the compressible case. This is not an easy task due to the nonlinearities in the vorticity equation and to the lack of the incompressibility of the velocity. We recall that the vorticity $\omega_\varepsilon=\partial_1 v^2_\varepsilon-\partial_2 v^1_\varepsilon$ satisfies the compressible transport equation:
$$
\partial_t\omega_\varepsilon+v_\varepsilon\cdot\nabla\omega_\varepsilon+\omega_\varepsilon\textnormal{div }v_\varepsilon=0.
$$
We will prove in Theorem \ref{theo3} the following   key estimate: for $p\in[1,+\infty[$
$$\Vert\omega_{\varepsilon}(t)\Vert_{B_{\infty,1}^{0}}\le C\Vert\omega_{0,\varepsilon}\Vert_{B_{\infty,1}^{0}}\Big(1+e^{C\Vert\nabla v_{\varepsilon} \Vert_{L_{t}^{1}L^\infty}}\Vert\textnormal{div}\,v_{\varepsilon}\Vert^{2}_{L_{t}^{1}B_{p,1}^{\frac2p}}\Big) \Big(1+\int_{0}^{t}\Vert\nabla v_{\varepsilon}(\tau)\Vert_{L^{\infty}}d\tau\Big).$$
We observe that for  $\textnormal{div}\,v_{\varepsilon}=0$ we get the Vishik's estimate. 
The proof will be done in the spirit of the work \cite{hmid-kera} but with slight modifications. In the first step, we use the Lagrangian coordinates combined with a filtration procedure to get rid of the compressible part. In the second step we establish  a suitable splitting of the vorticity and   use a dynamical interpolation method.

Concerning the incompressible limit we prove first the strong convergence in $L^\infty_{\textnormal{loc}}(\RR_+;L^2)$ by using  Strichartz estimates. This leads according to some standard  interpolation inequalities to the strong convergence in Besov spaces $B_{2,1}^s$, for all $s<2.$  However for the strong convergence in the space of the initial data $B_{2,1}^2$, we establish an additional frequency decay of the energy uniformly with respect to $t$ and $\varepsilon$. In other words, from the assumption on the initial data we can find a nondecreasing function $\Psi$ independent of $\varepsilon$ such that $( v_{0,\varepsilon}, c_{0,\varepsilon})$ belongs to $B_{2,1}^{2,\Psi},$ see the Definition \ref{heter}. This regularity will be conserved in the full interval $[0,T_\varepsilon]$ and thus the energy transfer for higher frequencies can not occur.

The paper is organized as follows: we recall in section $2$ some basic results about the Littlewood-Paley operators and Besov spaces. In section $3$, we gather some energy estimates for the heterogeneous Besov spaces $B_{2,1}^{s,\Psi}$ and we establish some useful Strichartz estimates for the acoustic parts of the fluid. In section $4$, we  establish a logarithmic estimate for a compressible transport model. In section $5$, we generalize the result of  Theorem \ref{theo2} and we give the proofs. 
\section{Basic Tools}
In this preliminary section, we are going to recall the  Littlewood-Paley operators and give some of their elementary properties. We will also  introduce some function spaces and review some important lemmas that will be used later.\\ 

\subsection{Littlewood-Paley operators}
We denote by $C$ any positive constant that may change from line to line and $C_{0}$ a real positive constant depending on the size of the initial data. We will use the following notations: for any positive $A$ and $B$, the notation  $A\lesssim B$ means that there exists a positive constant $C$ independent on $A$ and $B$ and such that $A\leqslant CB$.\\
To define Besov spaces we need to the following dyadic unity partition (see \cite{che}). There are  
two nonnegative radial functions $\chi\in\mathcal{D}(\mathbb{R}^{2})$ and $\varphi\in\mathcal{D}(\mathbb{R}^{2}\backslash\{ 0\})$ such that 
$$\chi(\xi)+ \displaystyle \sum_{q\ge 0}\varphi(2^{-q}\xi)=1, \quad\forall \xi\in\mathbb{R}^{2},$$
$$\displaystyle\sum_{q\in\mathbb{Z}}\varphi(2^{-q}\xi)=1, \quad\forall \xi\in\mathbb{R}^{2}\backslash\{0\},$$
$$\vert p-q\vert\ge 2\Rightarrow\mbox{supp }{\varphi}(2^{-p}\cdot)\cap\mbox{supp }{\varphi}(2^{-q}\cdot)=\varnothing,$$
$$q\ge 1\Rightarrow \mbox{supp }{\chi}\cap\mbox{supp }{\varphi}(2^{-q}\cdot)=\varnothing.$$
Let $u\in\mathcal{S}^{\prime}(\RR^2)$ we define the  Littlewood-Paley operators by
\begin{eqnarray*}
\Delta_{-1}u=\chi(\DD)u,\;\;\forall q\ge 0,\;\;\Delta_{q}u=\varphi(2^{-q}\DD)u\;\;\textnormal{and}\;\;S_{q}u=\displaystyle \sum_{-1\le p\le q-1}\Delta_{p}u.
\end{eqnarray*}
We can  easily check that $$u=\sum_{q\in \ZZ}\Delta_{q}u,\;\;\forall u \in \mathcal{S}^{\prime}(\RR^{2}).$$
Moreover, the Littlewood-Paley decomposition satisfies the property of almost orthogonality: for any $u,v\in\mathcal{S}^{\prime}(\RR^2),$
$$\Delta_{p}\Delta_{q}u=0\qquad \textnormal{if} \qquad \vert p-q \vert \geqslant 2 \qquad$$
$$\Delta_{p}(S_{q-1}u\Delta_{q}v)=0 \qquad \textnormal{if} \qquad  \vert p-q \vert \geqslant 5.$$\\
Let us note that the above operators $\Delta_{q}$ and $S_{q}$ map continuously $L^{p}$ into itself uniformly with respect to  $q$ and $p$. 
By the same way we define the homogeneous operators:
$$
\forall q\in\mathbb Z,\,\quad\dot{\Delta}_{q}u=\varphi(2^{-q}\hbox{D})v\quad\hbox{and}\quad\dot{S}_{q}u=\sum_{j\leq q-1}\dot{\Delta}_{j}u. 
$$
We notice that these operators are of convolution type. For example for $q\in\ZZ,\,$ we have
$$\dot\Delta_{q}u=2^{2q}h(2^q\cdot)\ast u,\quad\hbox{with}\quad h\in\mathcal{S},\quad \widehat{h}(\xi)=\varphi(\xi).
$$
Now we recall Bernstein inequalities, see for exapmle \cite{che}.
\begin{lem}\label{ber}
There exists a constant $C>0$ such that for all $q\in\NN\,,\,k \in \NN$ and for every tempered distriubution $u$ we have  
\begin{eqnarray*}
\sup_{\vert\alpha\vert=k}\Vert\partial^{\alpha}S_{q}u\Vert_{L^{b}}\leqslant C^{k}2^{q\big(k+2\big(\frac{1}{a}-\frac{1}{b}\big)\big)}\Vert S_{q}u\Vert_{L^{a}}\quad \textnormal{for}\quad \; b\geqslant a\geqslant 1\\
C^{-k}2^{qk}\Vert{\Delta}_{q}u\Vert_{L^{a}}\leqslant \sup_{\vert\alpha\vert=k}\Vert\partial^{\alpha}{\Delta}_{q}u\Vert_{L^{a}}\leqslant C^{k}2^{qk}\Vert {\Delta}_{q}u\Vert_{L^{a}}.
\end{eqnarray*}
\end{lem}

\subsection{Besov spaces}
Now we will  define the  Besov spaces by using Littlewood-Paley operators. 
 Let $(p,r)\in[1,+\infty]^2$ and $s\in\mathbb R,$ then the non-homogeneous   Besov 
\mbox{space $B_{p,r}^s$} is 
the set of tempered distributions $u$ such that
$$
\|u\|_{B_{p,r}^s}:=\Big( 2^{qs}
\|\Delta_q u\|_{L^{p}}\Big)_{\ell^{r}}<+\infty.
$$
The homogeneous Besov space $\dot B_{p,r}^s$ is defined as the set of  $u\in\mathcal{S}'(\RR^d)$ up to polynomials such that
$$
\|u\|_{\dot B_{p,r}^s}:=\Big( 2^{qs}
\|\dot\Delta_q u\|_{L^{p}}\Big)_{\ell ^{r}(\ZZ)}<+\infty.
$$
We remark that the usual Sobolev space $H^s$ coincides with  $B_{2,2}^s$ for $s\in\RR$  and the H\"{o}lder space $C^s$ coincides with $B_{\infty,\infty}^s$ when $s$ is not an integer. 

The following embeddings are an easy consequence of  Bernstein inequalities, 
$$
B^s_{p_1,r_1}\hookrightarrow
B^{s+2({1\over p_2}-{1\over p_1})}_{p_2,r_2}, \qquad p_1\leq p_2\quad and \quad  r_1\leq r_2.
$$

{Let $T>0$} \mbox{and $\rho\geq1,$} we denote by $L^\rho_{T}B_{p,r}^s$ the space of tempered  distributions $u$ such that 
$$
\|u\|_{L^\rho_{T}B_{p,r}^s}:= \Big\|\Big( 2^{qs}
\|\Delta_q u\|_{L^p}\Big)_{\ell ^{r}}\Big\|_{L^\rho_{T}}<+\infty.$$
Now we will introduce the heterogeneous Besov spaces which are an extension of the classical Besov spaces.
\begin{defi}\label{heter}
Let $\Psi:\{-1\}\cup\NN\to\RR_+^*$ be a given function.

(i)
 We say that $\Psi$ belongs to the class $\mathcal{U}$ if the following conditions are satisfied:
\begin{enumerate}
\item $\Psi$ is a nondecreasing function.
\item There exists  $C>0$ such that
$$
\sup_{x\in\NN\cup\{-1\}}\frac{\Psi(x+1)}{\Psi(x)}\leq C.
$$

\end{enumerate}
(ii) We define the class $\mathcal{U}_\infty$ by the set of function  $\Psi\in \mathcal{U}$ satisfying 
$\displaystyle{\lim_{x\to +\infty}\Psi(x)=+\infty}$.

(iii) Let $s\in\RR, (p,r)\in [1,+\infty]^2$ and $\Psi\in \mathcal{U}$. We define the heterogeneous Besov space $B_{p,r}^{s,\Psi}$ as follows:
$$
u\in B_{p,r}^{s,\Psi} \Longleftrightarrow \|u\|_{B_{p,r}^{s,\Psi}}=\Big( \Psi(q)2^{qs} \|\Delta_q u\|_{L^p}\Big)_{\ell^r}<+\infty.
$$

\end{defi}
\begin{Rema}\label{comparison1}
\begin{enumerate}
\item From the part $(2)$ of the above Definition, we see that the profile $\Psi$ has at most an exponential growth: there exists $\alpha>0$ such that
$$
\Psi(-1)\le\Psi(q)\le \Psi(-1)e^{\alpha q},\quad\forall q\geq-1.
$$
When  the profile $\Psi$ has an exponential growth: $\Psi(q)=2^{\alpha q}$ with $\alpha\in\RR_+$,  the space $B_{p,r}^{s,\Psi}$ reduces to the classical Besov space $B_{p,r}^{s+\alpha}.$ 
\item The condition $(2)$ seems to be  necessary for the definition of $B_{p,r}^{s,\Psi}:$  it allows to get a coherent  definition which is independent of the choice of the dyadic partition. 

\end{enumerate}
\end{Rema}
The following lemma proved in \cite{Hmidi-comp} is important for the proof of Theorem \ref{theo2}. Roughly speaking, it says that  any element of a given Besov space is always more regular than its prescribed regularity. 

%
\begin{lem}\label{Besov}  
Let $s\in\RR,$ $p\in[1,+\infty],$ $r\in[1,+\infty[$ and $f\in B_{p,r}^{s}.$ Then there exists a function $\Psi\in\mathcal{U}_{\infty}$ such that $f\in B_{p,r}^{s,\Psi}.$
\end{lem}
This yields to the following result, see \cite{Hmidi-comp}.
\begin{cor}\label{cor1}
Let $s\in\RR, p\in[1,+\infty], r\in[1,+\infty[$ and $(g_\varepsilon)_{0<\varepsilon\le 1}$ be a family of smooth functions satisfying
$$\Big(\sum_{q\ge-1}2^{qsr}\sup_{0<\varepsilon\le 1}\Vert\Delta_{q}g_\varepsilon\Vert^{r}_{L^p}\Big)^{\frac{1}{r}}<+\infty.$$
Then there exists $\Psi\in \mathcal{U}_{\infty}$ such that 
$g_\varepsilon\in B_{p,r}^{s,\Psi},$ with uniform bounds, that is,
$$
\sup_{0<\varepsilon\le1} \Big(\sum_{q\ge-1}\Psi^{r}(q)2^{qsr}\Vert\Delta_{q}g_\varepsilon\Vert^{r}_{L^p}\Big)^{\frac{1}{r}}<+\infty.
$$
\end{cor}

For the proof of the next proposition, see for example \cite{dh04}.
\begin{lem}\label{propagation}
Let  $u$ be a smooth  vector field, not necessary of zero divergence. Let $f$ be a  smooth solution of the transport equation 
 $$
\partial_{t}f+u\cdot\nabla f=g,\, f_{|t=0}=f_0,
$$
such that $f_0\in B_{p,r}^s(\mathbb R^2)$ and 
$g\in{L^1_{\textnormal{loc}}}(\mathbb R_{+};B_{p,r}^{s}).  $     
{Then the following assertions hold true}
\begin{enumerate}
\item Let $(p,r)\in[1,\infty]^2$and $s\in]0,1[$, then
\begin{equation*}\label{df}
\|f(t)\|_{B_{p,r}^s}     
\leq 
Ce^{CV(t)}
\Big(\|f_0\|_{B_{p,r}^s}+\int_{0}^te^{-CV(\tau)}\|g(\tau)\|_{B_{p,r}^{s}}d\tau\Big),
\end{equation*}
where $ V(t)=\int_{0}^t\|\nabla u(\tau)\|_{L^\infty}d\tau$ and $C$ is a constant depending  on $s.$
 \item Let $s\in]-1,0], r\in[1,+\infty]$ and $ p\in[2,+\infty]$ with $s+\frac2p>0$, then 
 \begin{equation*}\label{df}
\|f(t)\|_{B_{p,r}^s}     
\leq 
Ce^{CV_p(t)}
\Big(\|f_0\|_{B_{p,r}^s}+\int_{0}^te^{-CV_p(\tau)}\|g(\tau)\|_{B_{p,r}^{s}}d\tau\Big),
\end{equation*}
with  $ V_p(t)=\|\nabla u\|_{L^1_tL^\infty}+\|\textnormal{div} u\|_{L^1_tB_{p,\infty}^{\frac2p}}.$ 
\end{enumerate}
\end{lem}
Now, we prove the following result which will be needed later.
\begin{lem}\label{le5} 
Let $v$ be a vector field such that $v\in B_{\infty,1}^{1}\cap L^{2}$ and $\omega$ its vorticity. Then we have
$$\Vert\nabla v\Vert_{L^\infty}\lesssim\Vert v\Vert_{L^2}+\Vert\textnormal{div}\,v\Vert_{B_{\infty,1}^{0}}+\Vert\omega\Vert_{B_{\infty,1}^{0}}.$$
\end{lem}
\begin{proof}
We split the velocity into incompressible and compressible parts: $v=\mathbb{P}v+\mathbb{Q}v.$ Then  we have 
$$\textnormal{curl}\,v=\textnormal{curl}\,\mathbb{P}\,v$$ 

We have by Bernstein ineguality, the continuity of $\dot{\Delta}_{q}\mathbb{P}:L^p\to L^p,$ $\forall p\in[1,\infty]$ uniformly in $q$ and $\Vert\dot{\Delta}_{q}v\Vert_{L^\infty}\sim 2^{-q}\Vert\dot{\Delta}_{q}\omega\Vert_{L^\infty},$ that
\begin{eqnarray}\label{homo}
\nonumber \Vert\nabla\mathbb{P}v\Vert_{L^\infty}&\le& \Vert\dot{\Delta}_{-1}\nabla\mathbb{P}v\Vert_{L^\infty} +\sum_{q\in\NN}\Vert\dot{\Delta}_{q}\nabla\mathbb{P}v\Vert_{L^\infty}\\
\nonumber&\lesssim& \Vert\dot{\Delta}_{-1}\mathbb{P}v\Vert_{L^2} +\sum_{q\in\NN}2^{q}\Vert\dot{\Delta}_{q}\mathbb{P}v\Vert_{L^\infty}\\
\nonumber&\lesssim& \Vert v\Vert_{L^2}+\sum_{q\in\NN}\Vert\Delta_{q}\omega\Vert_{L^\infty}\\
&\lesssim& \Vert v\Vert_{L^2}+\Vert\omega\Vert_{B_{\infty,1}^{0}}.
\end{eqnarray}
On the other hand, using Bernstein inequality leads to
\begin{eqnarray*}
\Vert\nabla \mathbb{Q}v\Vert_{L^\infty}
&=&\|\nabla^2\Delta^{-1}\textnormal{div } v \Vert_{L^\infty}\\
&\leq& \|\Delta_{-1}\nabla^2\Delta^{-1}\textnormal{div } v \Vert_{L^\infty}+\sum_{q\geq0}\|\Delta_q\nabla^2\Delta^{-1}\textnormal{div } v \Vert_{L^\infty}\\
&\lesssim& \|v\|_{L^2}+\|\textnormal{div} \,v\|_{B_{\infty,1}^0}.
\end{eqnarray*}
Combining this estimate with  \eqref{homo}, gives
\begin{eqnarray*}
\Vert\nabla v\Vert_{L^\infty}
\lesssim \Vert v\Vert_{L^2}+\Vert\textnormal{div}\,v\Vert_{B_{\infty,1}^{0}}+\Vert\omega\Vert_{B_{\infty,1}^{0}}.
\end{eqnarray*}
This achieves the proof of the lemma.
\end{proof}

\section{Preliminaries}
This section is devoted to some  useful estimates for the system \eqref{C1} that will be crucial for the proof of the main results. We will discuss  some energy estimates for the full system \eqref{C1}   and  give some Strichartz estimates for the acoustic operator.
\subsection{Energy estimates}
Here we list two energy estimates for \eqref{C1}. The first one  is very classical and is concerned with the $L^2$-estimate. However the second one deals with the energy estimates in the heterogeneous Besov space $B_{2,1}^{s,\Psi}.$ All the spaces of the initial data are constructed over the Lebesgue space $L^2$ in order to remove the singular  terms  and get 	 uniform estimates with respect to the Mach number $\varepsilon.$ For the proof, see for instance \cite{Hmidi-comp}.
\begin{prop}\label{f1}
Let $(\vepsilon,\cepsilon)$  be a smooth solution of \eqref{C1}  and $\Psi\in \mathcal{U}$, see the Definition \ref{heter}. 
\begin{enumerate}
\item $L^2-$estimate: there exists $C>0$ such that $\forall t\geq0$
$$
\|(\vepsilon,\cepsilon)(t)\|_{L^2}\le C\|(\vepsiloninit,\cepsiloninit)\|_{L^2}e^{C\|\divergence\vepsilon\|_{L^1_tL^\infty}}.
$$
\item Besov estimates: for $s>0$ and $r\in[1,+\infty],$ there exists $C>0$ such that
$$
\|(\vepsilon,\cepsilon)(t)\|_{B_{2,r}^{s,\Psi}}\le C\|(\vepsiloninit,\cepsiloninit)\|_{B_{2,r}^{s,\Psi}}\,e^{CV_\varepsilon(t)}
$$ 
with
$$V_\varepsilon(t):=\|\nabla\vepsilon\|_{L^1_tL^\infty}+\|\nabla \cepsilon\|_{L^1_tL^\infty}. 
$$
\end{enumerate}
\end{prop}
We emphasize that for $\Psi\equiv1$ the space $B_{2,r}^{s,\Psi}$ reduces to the classical Besov space $B_{2,r}^{s}.$ Therefore we get the known energy estimate, see for instance \cite{dh04},
\begin{equation}\label{22-fev}
\|(\vepsilon,\cepsilon)(t)\|_{B_{2,r}^{s}}\le C\|(\vepsiloninit,\cepsiloninit)\|_{B_{2,r}^{s}}\,e^{CV_\varepsilon(t)}.
\end{equation}  
\subsection{Strichartz estimates}
Our goal is to  establish some classical  Strichartz estimates for the acoustic parts which are governed by wave equations. Roughly speaking, we will show that  the averaging in time of the compressible part of the velocity $v_\varepsilon$ and the sound speed $c_\varepsilon$ will lead to vanishing quantities  when $\varepsilon$ goes to zero. We point out that  Ukai \cite{u86} used this fact but under its dispersion form  in order to study the incompressible limit in the framework of the ill-prepared initial data. There are various ways to exhibit the wave structure and we prefer here the complex version.
We start with rewriting the system \eqref{C1} under the form:
\begin{equation}\label{t3} 
\left\{ \begin{array}{l} 
\partial_{t}v_{\varepsilon}+\frac{1}{\varepsilon}\nabla c_{\varepsilon}=-v_{\varepsilon}\cdot\nabla v_{\varepsilon}-\bar\gamma c_{\varepsilon}\nabla c_{\varepsilon}:=f_{\varepsilon}\\ 
\partial_{t}c_{\varepsilon}+\frac{1}{\varepsilon}\textnormal{div}\,v_{\varepsilon}=-v_{\varepsilon}\cdot\nabla c_{\varepsilon}-\bar\gamma c_{\varepsilon}\textnormal{div}\,v_{\varepsilon}:=g_{\varepsilon}\\
(v_{\varepsilon},c_{\varepsilon})_{| t=0}=(v_{0,\varepsilon},c_{0,\varepsilon}).\\  
\end{array} \right.
\end{equation}
Denote by $\mathbb{Q}v:=\nabla\Delta^{-1}\textnormal{div}\,v$ the compressible part of the velocity $v.$ Then the quantity
$$\Gamma_{\varepsilon}:=\mathbb{Q}v_{\varepsilon}-i\nabla\vert\DD\vert^{-1}c_{\varepsilon},\qquad\textnormal{with}\;\;\vert\DD\vert=(-\Delta)^{\frac12}$$
satisfies the following wave equation
\begin{equation}\label{t4}
(\partial_{t}+\frac{i}{\varepsilon}\vert\DD\vert)\Gamma_{\varepsilon}=\mathbb{Q}f_{\varepsilon}-i\nabla\vert\DD\vert^{-1}g_{\varepsilon}.
\end{equation}
Similarly, we can easily check that the quantity
$$\Upsilon_{\varepsilon}:=\vert\DD\vert^{-1}\textnormal{div}\,v_{\varepsilon}+i{ c}_{\varepsilon}$$
obeys to the wave equation
\begin{equation}\label{t5}
(\partial_{t}+\frac{i}{\varepsilon}\vert\DD\vert)\Upsilon_{\varepsilon}=\vert\DD\vert^{-1}\textnormal{div}\,f_{\varepsilon}+i g_{\varepsilon}.
\end{equation}
Now, we are in a position  to use of the following Strichartz estimates stated in space dimension two and for the proof see for instance \cite{bcd11,dh04,gv95}.
\begin{lem}\label{le1} Let $\psi$ be a  solution of the wave equation
$$(\partial_{t}+\frac{i}{\varepsilon}\vert\DD\vert)\psi=G,\qquad \psi_{| t=0}=\psi_{0}.$$
Then there exists an absolute  constant $C$  such that for every $T>0,$
$$\Vert\psi\Vert_{L_{T}^{r}L^p}\le C\varepsilon^{\frac 14-\frac{1}{2p}}\big(\Vert\psi_{0}\Vert_{\dot{B}_{2,1}^{\frac34-\frac{3}{2p}}}+\Vert G\Vert_{L_{T}^{1}\dot{B}_{2,1}^{\frac34-\frac{3}{2p}}}\big),$$
 for all $p\in[2,+\infty]$ and $r=4+\frac{8}{p-2}$.
\end{lem}
This result  enables  us to get the following estimates.
\begin{prop}\label{pr2}
Let $(v_{0,\varepsilon},c_{0,\varepsilon})$ be a bounded family in $ B_{2,1}^{\frac{7}{4}}$. Then any smooth solution of \eqref{C1} defined in the time interval $[0,T]$ satisfies,
$$\Vert(\mathbb{Q}v_{\varepsilon}, c_{\varepsilon})\Vert_{L_{T}^{4}L^\infty}\le C_{0}\varepsilon^{\frac 14}(1+T)e^{CV_{\varepsilon}(T)} ,$$
where  $C_{0}$ depends on the quantities $\displaystyle{\sup_{0<\varepsilon\le1}\Vert(v_{0,\varepsilon},c_{0,\varepsilon})\Vert_{B_{2,1}^{\frac 74}}}$ and 
$$V_{\varepsilon}(T):=\Vert\nabla v_{\varepsilon}\Vert_{L_{T}^{1}L^\infty}+\Vert\nabla c_{\varepsilon}\Vert_{L_{T}^{1}L^\infty}.$$
\end{prop}
\begin{proof} 
We apply Lemma \ref{le1} to the equation \eqref{t4} with $p=+\infty,$ then
$$\Vert\Gamma_{\varepsilon}\Vert_{L_{T}^{4}L^\infty}\lesssim\varepsilon^{\frac 14}\big(\Vert\Gamma_{\varepsilon}^{0}\Vert_{\dot{B}_{2,1}^{\frac 34}}+\Vert\mathbb{Q}f_{\varepsilon} -i\nabla\vert\DD\vert^{-1}g_{\varepsilon}\Vert_{L_{T}^{1}\dot{B}_{2,1}^{\frac 34}}\big).$$
Using the continuity of the Riesz operators $\mathbb{Q}$ and $\nabla\vert\DD\vert^{-1}$ on the homogeneous Besov spaces, combined with the embedding $B_{2,1}^{\frac 34}\hookrightarrow\dot{B}_{2,1}^{\frac 34}$ and H\"older inequality
\begin{eqnarray}\label{t6}
\nonumber \Vert\Gamma_{\varepsilon}\Vert_{L_{T}^{4}L^\infty}&\lesssim& \varepsilon^{\frac{1}{4}}\Big(\Vert(v_{0,\varepsilon},c_{0,\varepsilon})\Vert_{\dot{B}_{2,1}^{\frac 34}}+\Vert(f_{\varepsilon},g_{\varepsilon})\Vert_{L_{T}^{1}\dot{B}_{2,1}^{\frac{3}{4}}}\Big)\\
&\lesssim& \varepsilon^{\frac{1}{4}}\Big(\Vert(v_{0,\varepsilon},c_{0,\varepsilon})\Vert_{B_{2,1}^{\frac 34}}+T\Vert(f_{\varepsilon},g_{\varepsilon})\Vert_{L_{T}^{\infty}B_{2,1}^{\frac{3}{4}}}\Big).
\end{eqnarray}
It remains to estimate $\Vert(f_{\varepsilon},g_{\varepsilon})\Vert_{L_{T}^{\infty}B_{2,1}^{\frac{3}{4}}}.$ For this purpose we use the following law product which can be proved for example by using Bony's decomposition :
$$\Vert f\cdot\nabla g\Vert_{B_{2,1}^{\frac{3}{4}}}\lesssim\Vert f\Vert_{L^\infty}\Vert g\Vert_{B_{2,1}^{\frac{7}{4}}}+\Vert g\Vert_{L^\infty}\Vert f\Vert_{B_{2,1}^{\frac{7}{4}}}.$$
Thus we get
\begin{eqnarray*}
\Vert(f_{\varepsilon},g_{\varepsilon})\Vert_{B_{2,1}^{\frac{3}{4}}}&\lesssim& \Vert(v_{\varepsilon},c_{\varepsilon})\Vert_{L^\infty}\Vert(v_{\varepsilon},c_{\varepsilon}) \Vert_{B_{2,1}^{\frac{7}{4}}}\\
&\lesssim& \Vert(v_{\varepsilon},c_{\varepsilon})\Vert^{2}_{B_{2,1}^{\frac{7}{4}}},
\end{eqnarray*}
where we have used in the last line the embedding $B_{2,1}^{\frac{7}{4}}\hookrightarrow L^\infty.$ Therefore \eqref{22-fev} \mbox{yields to}
$$\Vert(f_{\varepsilon},g_{\varepsilon})\Vert_{L_{T}^{\infty}B_{2,1}^{\frac{3}{4}}}\lesssim\Vert(v_{0,\varepsilon},c_{0,\varepsilon})\Vert^{2}_{B_{2,1}^{\frac{7}{4}}}e^{CV_{\varepsilon}(T)}.$$
Plugging this inequality into \eqref{t6}, we obtain
\begin{eqnarray*}
\Vert\Gamma_{\varepsilon}\Vert_{L_{T}^{4}L^\infty}&\lesssim& \varepsilon^{\frac{1}{4}}\big(\Vert(v_{0,\varepsilon},c_{0,\varepsilon})\Vert_{B_{2,1}^{\frac{3}{4}}}+T\Vert(v_{0,\varepsilon},c_{0,\varepsilon})\Vert^{2}_{B_{2,1}^{\frac{7}{4}}}e^{CV_{\varepsilon}(T)}\big)\\
&\lesssim& C_{0}\varepsilon^{\frac{1}{4}}(1+T)e^{CV_{\varepsilon}(T)}.
\end{eqnarray*}
Remark that the real part of $\Gamma_{\varepsilon}$ is the compressible part of the velocity $v_{\varepsilon},$ then
$$\Vert\mathbb{Q}v_{\varepsilon}\Vert_{L_{T}^{4}L^\infty}\lesssim C_{0}\varepsilon^{\frac{1}{4}}(1+T)e^{CV_{\varepsilon}(T)}.$$
We can prove in the similar way that,
$$\Vert\Upsilon_{\varepsilon}\Vert_{L_{T}^{4}L^\infty}\lesssim C_{0}\varepsilon^{\frac{1}{4}}(1+T)e^{CV_{\varepsilon}(T)}$$
and therefore
$$\Vert c_{\varepsilon}\Vert_{L_{T}^{4}L^\infty}\lesssim C_{0}\varepsilon^{\frac{1}{4}}(1+T)e^{CV_{\varepsilon}(T)}.$$
The proof of the Proposition \ref{pr2} is now achieved.
\end{proof}

\section{ Logarithmic estimate} The purpose of this paragraph is to study the linear growth of the norm $B_{\infty,1}^{0}$ for the following compressible transport model,
\begin{equation}\label{x8} 
\left\{ \begin{array}{l}
\partial_t f+v\cdot\nabla f+f\textnormal{div}\,v=0\\
f_{| t=0}=f_{0}.  
\end{array} \right.
\end{equation} 
We point out that the dynamics of the vorticity  for the system \eqref{C1} is governed by an equation of type (\ref{x8}). In the framework of the incompressible vector fields, that is $\textnormal{div }v=0$, Vishik \cite{vis} established the following linear growth for Besov regularity with zero index,
$$
\|f(t)\|_{B_{\infty,1}^0}\le C\|f_0\|_{B_{\infty,1}^0}\Big( 1+\int_0^t\|\nabla v(\tau)\|_{L^\infty}d\tau\Big).
$$
A simple proof for this result was given in \cite{hmid-kera} and where we extended  the Vishik's result to a transport-diffusion model. Our main goal here is to valid the previous linear growth  to the compressible model and this is   very crucial for the proof of the Theorem \ref{theo2}. 
Our result reads as follows, 
\begin{thm}\label{theo3}
Let $v$ be a smooth vector field and $f$ be a smooth solution of the transport \mbox{equation \eqref{x8}.}
Then for every $1\le p<+\infty$, there exists a constant $C$ depending  only on $p$  such that
$$\|f(t)\|_{B_{\infty,1}^0}\le C \|f_0\|_{B_{\infty,1}^0}\Big(1+{e^{C\|\nabla v\|_{L^1_tL^\infty}}\|\textnormal{div}\,v\|_{L^1_tB_{p,1}^{\frac2p}}^2}\Big)\Big(1+\int_0^t\|\nabla v(\tau)\|_{L^\infty}d\tau\Big).$$
\end{thm}
\begin{proof}
First we observe that the above estimate reduces in the incompressible case to the Vishik's estimate. Here, we will try to use the same approach of \cite{hmid-kera}, but as we will see the lack of the incompressibility brings more technical difficulties. 
Let us  denote by $\psi$ the flow associated to the velocity $v:$ 
$$\psi(t,x)=x+\int_{0}^{t}v\big(\tau,\psi(\tau,x)\big)d\tau,\;\;(t,x)\in\RR_+\times\RR^{2}.$$
We will make use of the following estimate for the flow and its inverse $\psi^{-1}$
\begin{equation}\label{eq1}
\Big\Vert\nabla\big(\psi(\tau,\psi^{-1}(t,\cdot))\big)\Big\Vert_{L^\infty}\le e^{\big\vert\int_{\tau}^{t}\Vert\nabla v(s)\Vert_{L^\infty}ds\big\vert}.
\end{equation}
We set $g(t,x)=f(t,\psi(t,x))$ then it is clear that $g$ satisfies the equation
$$\partial_t g(t,x)+(\textnormal{div}\,v)(t,\psi(t,x))g(t,x)=0.$$
Therefore we obtain
$$g(t,x)=f_0(x)e^{-\int_0^t(\textnormal{div}\,v)(\tau,\psi(\tau,x))d\tau}.$$
It follows that
$$f(t,x)=f_0(\psi^{-1}(t,x))e^{-\int_0^t(\textnormal{div}\,v)(\tau,\psi(\tau,\psi^{-1}(t,x)))d\tau}.$$
Set $h(x)=e^{x}-1$ and $W(t,x)=-\int_0^t(\textnormal{div}\,v)(\tau,\psi(\tau,\psi^{-1}(t,x)))d\tau,$ then we infer
$$f(t,x)=f_0(\psi^{-1}(t,x))\big(h(W(t,x))+1\big).$$
Thus the problem reduces  to the establishment of a composition law in $B_{\infty,1}^0$ space. For this purpose, we will use the following law product which can be obtained by using Bony's decomposition \cite{bo}: for $1\le p<\infty,$
$$\| uv\|_{B_{\infty,1}^0}\le C\|u\|_{B_{p,1}^{\frac2p}}\|v\|_{B_{\infty,1}^0}.$$
Therefore
\begin{equation}\label{log-hmd}
\|f(t)\|_{B_{\infty,1}^0}\le C\| f_0(\psi^{-1}(t,\cdot))\|_{B_{\infty,1}^0}\big(\|h(W(t,\cdot))\|_{B_{p,1}^{\frac2p}}+1\big).
\end{equation}
At this stage we need to the following result.
$$\|h(W(t,\cdot))\|_{B_{p,1}^{\frac2p}}\le C\|W(t,\cdot)\|_{B_{p,1}^{\frac2p}} e^{C\|W(t)\|_{L^\infty}}.$$
The proof of this estimate can be done as follows. By definition we get easily
%
$$\|h(W(t,\cdot))\|_{B_{p,1}^{\frac2p}}\le\sum_{n\ge 1}\frac{1}{n!}\|W^{n}(t,\cdot)\|_{B_{p,1}^{\frac2p}}.$$
According to the law product
$$
\|u^2\|_{B_{p,1}^{\frac2p}}\le C\|u\|_{L^\infty}\|u\|_{B_{p,1}^{\frac2p}}, p<+\infty,
$$
and using the induction principle we infer that for $n\geq 1$,
\begin{equation*}
\|W^{n}(t,\cdot)\|_{B_{p,1}^{\frac2p}}\le C^{n-1}\|W(t,\cdot)\|^{n-1}_{L^{\infty}}\|W(t,\cdot)\|_{B_{p,1}^{\frac2p}}.
\end{equation*}
Therefore 
\begin{eqnarray*}
\|h(W(t,\cdot))\|_{B_{p,1}^{\frac2p}}&\le& \Vert W\Vert_{B_{p,1}^{\frac2p}}\sum_{n\ge 1}\frac{C^{n-1}}{n!}\Vert W\Vert^{n-1}_{L^\infty}\\
&\le& \Vert W\Vert_{B_{p,1}^{\frac2p}}\sum_{n\ge 0}C^{n}\frac{\Vert W\Vert^{n}_{L^\infty}}{(n+1)!}\\
&\le& \Vert W\Vert_{B_{p,1}^{\frac2p}}\sum_{n\ge 0}C^{n}\frac{\Vert W\Vert^{n}_{L^\infty}}{n!}\\
&\le& \Vert W\Vert_{B_{p,1}^{\frac2p}}e^{C\Vert W\Vert_{L^\infty}}.
\end{eqnarray*}
This ends the proof of the desired inequality. Thus it follows that
\begin{eqnarray}\label{eq2}
\nonumber \|h(W(t,\cdot))\|_{B_{p,1}^{\frac2p}}&\lesssim& \|W(t,\cdot)\|_{B_{p,1}^{\frac2p}} e^{C\|W(t)\|_{L^\infty}}\\
&\lesssim& \|W(t,\cdot)\|_{B_{p,1}^{\frac2p}} e^{C\int_0^t\|\textnormal{div}\,v(\tau)\|_{L^\infty}d\tau}.
\end{eqnarray}
To estimate $W(t)$ we set $k_\tau(t,x)=\textnormal{div}\,v(\tau,\psi(\tau,\psi^{-1}(t,x))$, then $k_\tau(t,\psi(t,x))=\textnormal{div}\,v(\tau,\psi(\tau,x))$ and it follows that
\begin{equation*}
\left\{ 
\begin{array}{ll} 
\partial_{t}k_\tau+v\cdot\nabla k_\tau=0 \\ 
k_\tau(\tau,x)=\textnormal{div}\,v(\tau,x) .
\end{array} \right.
\end{equation*}
It remains to estimate $\Vert k_{\tau}(t)\Vert_{B_{p,1}^{\frac2p}}.$ For this aim, we use Lemma \ref{propagation}-(1), yielding for $p>2$ to 
\begin{eqnarray*}
\|k_\tau(t)\|_{B_{p,1}^{\frac2p}}
&\le& C \|\textnormal{div}\,v(\tau)\|_{B_{p,1}^{\frac2p}}e^{|\int_\tau^t\|\nabla v(s)\|_{L^\infty}ds|}.
\end{eqnarray*}
Combining the definition of $W$ with this latter estimate we get
\begin{eqnarray*}
\|W(t)\|_{B_{p,1}^{\frac2p}}&\le& \int_0^t\|k_\tau(t)\|_{B_{p,1}^{\frac2p}}d\tau\\
&\le& C\int_0^t\|\textnormal{div}\,v(\tau)\|_{B_{p,1}^{\frac2p}}e^{\int_\tau^t\|\nabla v(s)\|_{L^\infty}ds}d\tau.
\end{eqnarray*}
Plugging this estimate into \eqref{eq2} we find
$$\|h(W(t,\cdot))\|_{B_{p,1}^{\frac2p}}\le C{ e^{C\|\nabla v\|_{L^1_tL^\infty}}\int_0^t\|\textnormal{div}\,v(\tau)\|_{B_{p,1}^{\frac2p}}d\tau}.$$
Inserting this estimate in \eqref{log-hmd} we obtain
\begin{equation}\label{log1}
\|f(t)\|_{B_{\infty,1}^0}\le C\| f_0(\psi^{-1}(t,\cdot))\|_{B_{\infty,1}^0}\Big(1+{ e^{C\|\nabla v\|_{L^1_tL^\infty}}\int_0^t\|\textnormal{div}\,v(\tau)\|_{B_{p,1}^{\frac2p}}d\tau}\Big)
\end{equation}
It remains to estimate $\| f_0(\psi^{-1}(t,\cdot))\|_{B_{\infty,1}^0}$. Set $w(t,x)=f_0(\psi^{-1}(t,x))$ then it satisfies the transport equation
\begin{equation*}
\left\{ 
\begin{array}{ll} 
\partial_{t}w+v\cdot\nabla w=0 \\ 
w(0,x)=f_0(x). 
\end{array} \right.
\end{equation*}
We will split the  initial data into Fourier modes $\displaystyle{f_0=\sum_{q\geq-1}\Delta_q f_0}$ and for each frequency $q\geq-1,$ we denote by $\tilde{w}_{q}$ the unique global solution of the initial value problem
\begin{equation}\label{R_{Q}}\left\lbrace
\begin{array}{l}
\partial_t \tilde {w}_{q}+v\cdot\nabla \tilde{w}_{q}=0\\
{\tilde{w}_{q}}(0,\cdot)=\Delta_{q}f_0.\\
\end{array}
\right.
\end{equation}
According to Lemma \ref{propagation}-(2), we have 
$$\|\tilde{w}_{q}(t)\|_{B_{\infty,\infty}^{\pm s}}\leq C\|\Delta_{q} f_0\|_{B_{\infty,\infty}^{\pm s}} e^{CV(t)},\quad \forall \, 0<s<\min(1,\frac2p),$$
with
$\displaystyle{V(t):=\|\nabla v\|_{L^1_{t}L^\infty}+\|\textnormal{div}\,v\|_{L^1_tB_{p,\infty}^{\frac2p}}}.$
Combined with  the  definition of Besov spaces, this implies for $j,q\geq-1$
\begin{equation}\label{t7}
\|\Delta_{j}\tilde {w}_{q}(t)\|_{L^\infty}\leq C2^{-s|j-q|}\|\Delta_{q} f_0\|_{L^\infty} e^{CV(t)}.
\end{equation}
By linearity and the uniqueness of the solution, it is obvious that 
$$w=\sum_{q=-1}^\infty\tilde {w}_{q}.$$
From the definition of Besov spaces we have
\begin{eqnarray}
\label{t8}
\|w(t)\|_{B_{\infty,1}^0}\leq\sum_{|j-q|\geq N}\|\Delta_{j}\tilde{w}_{q}(t)\|_{L^\infty}+\sum_{|j-q|< N}\|\Delta_{j}\tilde{w}_{q}(t)\|_{L^\infty},
\end{eqnarray}
where $N\in \Bbb N$ is to be chosen later. To deal with the first sum we use (\ref{t7}) 
\begin{eqnarray}\label{t9}
\nonumber\sum_{|j-q|\geq N}\|\Delta_{j}\tilde{w}_{q}(t)\|_{L^\infty}&\lesssim& 2^{-N s}\sum_{q\geq-1}\|\Delta_{q}f_0\|_{L^\infty}e^{CV(t)}\\
&\lesssim&  2^{-Ns}\|f_0\|_{B_{\infty,1}^0}e^{CV(t)}.
\end{eqnarray}
For the second  sum in  the right-hand side of (\ref{t8}), we use the fact that the \mbox{operator $\Delta_{j}$} maps  $L^\infty$ into itself uniformly with respect to $j$,
\begin{equation*}\nonumber\sum_{|j-q|< N}
\|\Delta_{j}\tilde{w}_{q}(t)\|_{L^\infty}\lesssim \sum_{|j-q|< N}
\|\tilde{w}_{q}(t)\|_{L^\infty}.
\end{equation*}
Applying the maximum principle to  the system (\ref{R_{Q}}) yields 
$$\|\tilde{w}_{q}(t)\|_{L^\infty}\leq \|\Delta_q f_0\|_{L^\infty}.$$
So, it holds that 
\begin{equation*}
\sum_{|j-q|<N}\|\Delta_{j}\tilde{w}_{q}(t)\|_{L^\infty}\lesssim N\|f_0\|_{B_{\infty,1}^0}.
\end{equation*}
The outcome is  the following
$$\|w(t)\|_{B_{\infty,1}^0}\lesssim\|f_0\|_{B_{\infty,1}^0}\Big(2^{- Ns}e^{CV(t)}+N\Big).$$
Choosing $N$ such that
$$N=\Big[\frac{CV(t)}{ s\log 2}+1\Big],$$
we get 
$$\|w(t)\|_{B_{\infty,1}^0}\le C \|f_0\|_{B_{\infty,1}^0}\Big(1+\int_0^t\big(\|\nabla v(\tau)\|_{L^\infty}+\|\textnormal{div}\,v(\tau)\|_{B_{p,\infty}^{\frac2p}}\big)d\tau\Big).$$
Therefore
$$\|f_{0}\big(\psi^{-1}(t,\cdot)\big)\|_{B_{\infty,1}^0}\le C \|f_0\|_{B_{\infty,1}^0}\Big(1+\int_0^t\big(\|\nabla v(\tau)\|_{L^\infty}+\|\textnormal{div}\,v(\tau)\|_{B_{p,\infty}^{\frac2p}}\big)d\tau\Big)$$
Inserting this estimate into \eqref{log1} gives
$$\|f(t)\|_{B_{\infty,1}^0}\le C\|f_0\|_{B_{\infty,1}^0}\Big(1+{ e^{C\|\nabla v\|_{L^1_tL^\infty}}\|\textnormal{div}\,v\|_{L^1_tB_{p,1}^{\frac2p}}^2}\Big)\Big(1+\int_0^t\|\nabla v(\tau)\|_{L^\infty}d\tau\Big).$$
This is the desired result.
\end{proof}

\section{Proofs of the main results}
In this section we will firstly  extend  the result of Theorem \ref{theo2} to the framework of the heterogeneous Besov spaces $B_{2,1}^{s,\Psi}$. Secondly we will see how to derive the proof of the Theorem \ref{theo2} and ultimately we will devote the rest of the paper to the discussion of the proof of the general statement given in Theorem \ref{ttheo2}.
\subsection{General statement}
We will give  a generalization of the Theorem \ref{theo2}. 
\begin{thm}\label{ttheo2} Let $\Psi\in\mathcal{U}_{\infty}$ and $\{(v_{0,\varepsilon},c_{0,\varepsilon})\}_{0<\varepsilon\le 1}$ be a bounded family in $B_{2,1}^{2,\Psi}$ of initial data, that is
$$\sup_{0<\varepsilon\le 1}\Vert(v_{0,\varepsilon},c_{0,\varepsilon})\Vert_{B_{2,1}^{2,\Psi}}:=M_0<+\infty.$$
Then the system \eqref{C1} admits a unique solution $(v_{\varepsilon},c_{\varepsilon})\in\mathcal{C}([0,T_{\varepsilon}];B_{2,1}^{2,\Psi}),$  with
$$T_{\varepsilon}= C_{0}\log\log\,\{\Psi(\log({\varepsilon^{-1}}))\}.$$
Moreover, there exists $\eta>0$ such that   for small $\varepsilon$ and for all $0\le T\le T_{\varepsilon},$
$$\Vert(\textnormal{div}\,v_{\varepsilon},\nabla c_{\varepsilon})\Vert_{L_{T}^{1}L^\infty}\le C_{0}\Psi^{-\eta}(\varepsilon),\qquad\Vert\omega_{\varepsilon}(t)\Vert_{L^\infty}\le C_{0},$$ 
$$\Vert\nabla v_{\varepsilon}\Vert_{L_{T}^{1}L^\infty}\le C_{0}e^{C_{0}T}\qquad\textnormal{and}\;\;\;\Vert(v_{\varepsilon},c_{\varepsilon})(T)\Vert_{B_{2,1}^{2,\Psi}}\lesssim C_{0}e^{e^{C_{0}T}}.$$
The constant $C_0$ may depend on $M_0$ and on the profile $\Psi$.

Assume in addition that the incompressible parts $(\mathbb{P}v_{0,\varepsilon})_{0<\varepsilon\le1}$ converge in $L^2$ to some $v_{0}.$ Then the incompressible parts of the solution converge strongly to the global solution of the system \eqref{c2} with initial \mbox{data $v_0$.} More precisely, for all $T>0$ and for all $\widetilde\Psi\in \mathcal{U}$ such that
$$
\lim_{q\to+\infty}\frac{\widetilde\Psi(q)}{\Psi(q)}=0
$$
we have 
$$\lim_{\varepsilon\to0}\|\mathbb{P}v_{\varepsilon}- v\|_{L^\infty_TB_{2,1}^{2,\widetilde\Psi}}=0.$$
\end{thm}
\begin{Rema}
\begin{enumerate}  
\item{} We observe that the result for the subcritical regularities  is a special case of the above theorem. Indeed, let  $\{(v_{0,\varepsilon},c_{0,\varepsilon})\}_{0<\varepsilon\le 1}$ be a bounded family in $H^s, s>2$. It is plain that the following embedding holds true: $H^s\hookrightarrow B_{2,1}^{2,\Psi}, $ with $ \Psi(q)=2^{q\alpha}$ and $0<\alpha<s-2.$ Thus we get by applying Theorem \ref{ttheo2} that 
$$T_{\varepsilon}= C_{0}\log\log\,\{\frac{1}{\varepsilon}\}.$$
This result is in agreement with known results about the lifespan for the subcritical regularities.
\item{} If $\Psi$ has a polynomial growth: $\Psi(q)=(q+2)^\alpha, \alpha>0$ then it is easily seen   that $\Psi\in \mathcal{U}_{\infty}$ and 
$$T_\varepsilon= C_{0}\log\log\log\{\frac{1}{\varepsilon}\}.$$

\end{enumerate}
\end{Rema}  
Before going further into the details of the  proof of the Theorem \ref{ttheo2}, we will first show how to deduce the result of the  Theorem \ref{theo2}.
\subsection{Proof of Theorem \ref{theo2}} 
Since
$$\sum_{q\ge-1}2^{2q}\sup_{0<\varepsilon\le 1}\Vert(\Delta_{q}v_{0,\varepsilon},\Delta_{q}c_{0,\varepsilon})\Vert_{L^{2}}<+\infty,$$
we deduce from  Corollary \ref{cor1} the existence of a profile $\Psi\in\mathcal{U}_{\infty}$ such that the family $(v_{0,\varepsilon},c_{0,\varepsilon})_{0<\varepsilon\le 1}$ is uniformly bounded in $B_{2,1}^{2,\Psi}.$ Then according to Theorem \ref{ttheo2}, there exists a unique solution $(v_{\varepsilon},c_{\varepsilon})$ belonging to the space $\mathcal{C}\big([0,T_{\varepsilon}];B_{2,1}^{2,\Psi}\big)$ with $$T_\varepsilon=  C_{0}\log\log\,\{\Psi(\log({\varepsilon^{-1}}))\}.$$
Therefore it is clear that $(v_{\varepsilon},c_{\varepsilon})\in\mathcal{C}\big([0,T_{\varepsilon}];B_{2,1}^{2}\big)$ and 
$$\lim_{\varepsilon\to 0}T_{\varepsilon}=+\infty.$$
For the incompressible limit, we can apply the second part of Theorem \ref{ttheo2} with the profile $\widetilde\Psi$ equal to a nonnegative constant and thus the space $B_{2,1}^{2,\widetilde\Psi}$ reduces to the usual Besov space $B_{2,1}^2.$

We will now give the complete proof of Theorem \ref{ttheo2}, which will be done in several steps. We start first with estimating the lifespan of the solutions and we discuss at the end the incompressible limit problem.
\subsection{Lifespan of the solutions}
We will give an a priori bound of $T_\varepsilon$ and show that the acoustic parts vanish when the Mach number goes to zero.
\begin{proof}
Using Lemma \ref{le5}, we get
$$\Vert\nabla v_{\varepsilon}(t)\Vert_{L^{\infty}}\lesssim\Vert v_{\varepsilon}(t)\Vert_{L^{2}}+\Vert\textnormal{div}\,v_{\varepsilon}(t)\Vert_{B_{\infty,1}^{0}} +\Vert\omega_{\varepsilon}(t)\Vert_{B_{\infty,1}^{0}}.$$
Integrating in time and using Proposition \ref{f1} we obtain
\begin{eqnarray*}
\Vert\nabla v_{\varepsilon}\Vert_{L_{T}^{1}L^{\infty}}&\lesssim& \Vert v_{\varepsilon}\Vert_{L_{T}^{1}L^{2}}+\Vert\textnormal{div}\,v_{\varepsilon}\Vert_{L_{T}^{1}B_{\infty,1}^{0}} +\Vert\omega_{\varepsilon}\Vert_{L_{T}^{1}B_{\infty,1}^{0}}\\
&\lesssim& C_{0}T e^{C\Vert\textnormal{div}\,v_{\varepsilon}\Vert_{L_{T}^{1}L^{\infty}}}+\Vert\textnormal{div}\,v_{\varepsilon}\Vert_{L_{T}^{1}B_{\infty,1}^{0}} +\Vert\omega_{\varepsilon}\Vert_{L_{T}^{1}B_{\infty,1}^{0}}\\
&\lesssim& C_{0}(1+T)e^{C\Vert\textnormal{div}\,v_{\varepsilon}\Vert_{L_{T}^{1}B_{\infty,1}^{0}}} +\Vert\omega_{\varepsilon}\Vert_{L_{T}^{1}B_{\infty,1}^{0}}.
\end{eqnarray*}
Now we recall the following notation,
$$V_{\varepsilon}(T)=\Vert\nabla v_{\varepsilon}\Vert_{L_{T}^{1}L^\infty}+\Vert\nabla c_{\varepsilon}\Vert_{L_{T}^{1}L^\infty}.$$
Then it follows that
\begin{eqnarray}\label{g1}
V_{\varepsilon}(T)
&\le& C_{0}(1+T) e^{C\Vert(\textnormal{div}\,v_{\varepsilon},\nabla c_{\varepsilon})\Vert_{L_{T}^{1}B_{\infty,1}^{0}}} +\Vert\omega_{\varepsilon}\Vert_{L_{T}^{1}B_{\infty,1}^{0}}.
\end{eqnarray}
Applying Theorem \ref{theo3} gives
\begin{eqnarray}\label{e4}
\nonumber \Vert\omega_{\varepsilon}(t)\Vert_{B_{\infty,1}^{0}}&\le&C \Vert\omega_{\varepsilon}^{0}\Vert_{B_{\infty,1}^{0}}\Big(1+e^{C\Vert\nabla v_{\varepsilon}\Vert_{L_{t}^{1}L^{\infty}}} \Vert\textnormal{div}\,v_{\varepsilon}\Vert_{L_{t}^{1}B_{p,1}^{\frac{2}{p}}}^{2}\Big)(1+\Vert\nabla
v_{\varepsilon}\Vert_{L_{t}^{1}L^{\infty}})\\
&\le& C_{0}\big(1+e^{CV_{\varepsilon}(t)}\Vert\textnormal{div}\,v_{\varepsilon}\Vert_{L_{t}^{1}B_{p,1}^{\frac{2}{p}}}^{2}\big)\big(1+V_{\varepsilon}(t)\big).
\end{eqnarray}
To estimate $\Vert\textnormal{div}\,v_{\varepsilon}\Vert_{B_{p,1}^{\frac2p}},$ we use the following interpolation inequality: for $2<p<\infty,$ 
$$\Vert\textnormal{div}\,v_{\varepsilon}\Vert_{B_{p,1}^{\frac2p}}\lesssim \Vert\textnormal{div}\,v_{\varepsilon}\Vert^{\frac2p}_{B_{2,1}^{1}} \Vert\textnormal{div}\,v_{\varepsilon}\Vert^{1-\frac2p}_{B_{\infty,1}^{0}}.$$
Integrating in time and using Bernstein and H\"older inequalities, we obtain
\begin{eqnarray*}
\Vert\textnormal{div}\,v_{\varepsilon}\Vert_{L_{T}^{1}B_{p,1}^{\frac2p}}&\lesssim& T^{\frac2p}\Vert\textnormal{div}\,v_{\varepsilon}\Vert^{\frac2p}_{L_{T}^{\infty}B_{2,1}^{1}} \Vert\textnormal{div}\,v_{\varepsilon}\Vert^{1-\frac2p}_{L_{T}^{1}B_{\infty,1}^{0}}\\
&\lesssim& T^{\frac2p}\Vert v_{\varepsilon}\Vert^{\frac2p}_{L_{T}^{\infty}B_{2,1}^{2}}\Vert\textnormal{div}\,v_{\varepsilon}\Vert^{1-\frac2p}_{L_{T}^{1}B_{\infty,1}^{0}}.
\end{eqnarray*}
Choose $p=4$ and using Proposition \ref{f1}-(2) yields to
\begin{equation}\label{e5}
\Vert\textnormal{div}\,v_{\varepsilon}\Vert^{2}_{L_{T}^{1}B_{4,1}^{\frac12}}\lesssim C_{0}T e^{CV_{\varepsilon}(T)}\Vert\textnormal{div}\,v_{\varepsilon}\Vert_{L_{T}^{1}B_{\infty,1}^{0}}.
\end{equation}
Putting together \eqref{e4} and \eqref{e5} and integrating in time we find
$$\Vert\omega_{\varepsilon}\Vert_{L_{T}^{1}B_{\infty,1}^{0}}\lesssim C_{0}\int_{0}^{T}\big(1+te^{CV_{\varepsilon}(t)}\Vert\textnormal{div}\,v_{\varepsilon}\Vert_{L_{t}^{1}B_{\infty,1}^{0}}\big)(1+V_{\varepsilon}(t))dt.$$
Inserting this estimate into  \eqref{g1} 
\begin{equation}\label{g2}
V_{\varepsilon}(T)\le C_{0}(1+T) e^{C\Vert(\textnormal{div}\,v_{\varepsilon},\nabla c_{\varepsilon})\Vert_{L_{T}^{1}B_{\infty,1}^{0}}}+C_{0}\int_{0}^{T}\big(1+te^{CV_{\varepsilon}(t)}\Vert\textnormal{div}\,v_{\varepsilon}\Vert_{L_{t}^{1}B_{\infty,1}^{0}}\big)(1+V_{\varepsilon}(t))dt.
\end{equation}

It remains to estimate $\Vert(\textnormal{div}\,v_{\varepsilon},\nabla c_{\varepsilon})\Vert_{L^1_TB_{\infty,1}^{0}}.$ For this purpose we will develop an interpolation procedure between the Strichartz estimates for lower frequencies and the energy estimates for higher frequencies.  More precisely, let $N\in\NN^{*}$ that will be judiciously fixed later. Then using Bernstein inequality combined with  the continuity of the operator $\mathbb{Q}$ on the Lebesgue space $L^{2},$ we find
\begin{eqnarray*}
\Vert\textnormal{div}\,v_{\varepsilon}\Vert_{B_{\infty,1}^{0}}&=& \Vert\textnormal{div }\mathbb{Q}\,v_{\varepsilon}\Vert_{B_{\infty,1}^{0}}\\
&\le& \sum_{-1\le q\le N}\Vert\Delta_{q}\textnormal{div }\mathbb{Q}\,v_{\varepsilon}\Vert_{L^{\infty}} +\sum_{q>N}\Vert\Delta_{q}\textnormal{div }\mathbb{Q}\,v_{\varepsilon}\Vert_{L^{\infty}}\\
&\lesssim& \sum_{q\le N}2^{q}\Vert\Delta_{q}\mathbb{Q}v_{\varepsilon}\Vert_{L^\infty}+\frac{1}{\Psi(N)}\sum_{q>N}\Psi(q)2^{2q}\Vert\Delta_{q}\mathbb{Q}v_{\varepsilon}\Vert_{L^{2}}\\
&\lesssim& 2^{N}\Vert\mathbb{Q}\,v_{\varepsilon}\Vert_{L^\infty}+\frac{1}{\Psi(N)}\Vert v_{\varepsilon}\Vert_{B_{2,1}^{2,\Psi}}.
\end{eqnarray*}
We have used the fact that the profile $\Psi$ is a nondecreasing function. Now integrating in time and using Proposition \ref{pr2} and Proposition \ref{f1}-(2)
\begin{eqnarray*}
\Vert\textnormal{div}\,v_{\varepsilon}\Vert_{L_{T}^{1}B_{\infty,1}^{0}}&\lesssim& 2^{N}T^{\frac{3}{4}}\Vert\mathbb{Q}\,v_{\varepsilon}\Vert_{L_{T}^{4}L^\infty}+\frac{T}{\Psi(N)} \Vert v_{\varepsilon}\Vert_{L_{T}^{\infty}B_{2,1}^{2,\Psi}}\\
&\le& C_{0}\big(T^{\frac{3}{4}}(1+T)+T\big)e^{CV_{\varepsilon}(T)}\big(\varepsilon^{\frac{1}{4}}2^{N}+\frac{1}{\Psi(N)}\big)\\
&\le& C_{0}(1+T^{\frac{7}{4}})e^{CV_{\varepsilon}(T)}\Big(\varepsilon^{\frac{1}{4}}2^{N}+\frac{1}{\Psi(N)}\Big).
\end{eqnarray*}
By  similar computations we get
$$\Vert\nabla c_{\varepsilon}\Vert_{L_{T}^{1}B_{\infty,1}^{0}}\le C_{0}(1+T^{\frac{7}{4}})e^{CV_{\varepsilon}(T)}\Big(\varepsilon^{\frac{1}{4}}2^{N}+\frac{1}{\Psi(N)}\Big).$$
Therefore
$$\Vert(\textnormal{div}\,v_{\varepsilon},\nabla c_{\varepsilon})\Vert_{L_{T}^{1}B_{\infty,1}^{0}}\le C_{0}(1+T^{\frac{7}{4}})e^{CV_{\varepsilon}(T)}\Big(\varepsilon^{\frac{1}{4}}2^{N}+\frac{1}{\Psi(N)}\Big).$$
Now, we choose $N$ such that,
$$N\approx \log\big(\frac{1}{\varepsilon^{\frac{1}{8}}}\big).$$
Hence we get
\begin{eqnarray*}
\Vert(\textnormal{div}\,v_{\varepsilon},\nabla c_{\varepsilon}\Vert_{L_{T}^{1}B_{\infty,1}^{0}}\le C_{0}(1+T^{\frac{7}{4}})e^{CV_{\varepsilon}(T)}\big(\varepsilon^{\frac{1}{8}}+\frac{1}{\Psi\big(\log(\frac{1}{\varepsilon^{\frac{1}{8}}})\big)}\big).
\end{eqnarray*}
According to the Remark \ref{comparison1}, the profile $\Psi$ has at most an exponential growth: there exists $\alpha>0$ such that
$$
\forall x\geq-1;\quad\,\Psi(x)\le C e^{\alpha x}.
$$ 
Hence we get
\begin{equation}\label{choix1}
\varepsilon^{\frac{1}{8}}\le \frac{C^{\frac1\alpha}}{\Psi^{\frac1\alpha}(\log(\frac{1}{\varepsilon^{\frac{1}{8}}})}.
\end{equation}
Let $\beta=\min(1,\frac1\alpha)$ then 
\begin{equation}\label{e7}
\Vert(\textnormal{div}\,v_{\varepsilon},\nabla c_{\varepsilon}\Vert_{L_{T}^{1}B_{\infty,1}^{0}}
\le C_{0}(1+T^{\frac{7}{4}})e^{CV_{\varepsilon}(T)}\Phi(\varepsilon),
\end{equation}
with
$$
\Phi(\varepsilon):=\frac{1}{\Psi^\beta\big(\log(\frac{1}{\varepsilon^{\frac{1}{8}}})\big)}\cdot
$$
Plugging \eqref{e7} into \eqref{g2} and using Gronwall's inequality we obtain
\begin{eqnarray}\label{e8}
\nonumber V_{\varepsilon}(T)&\le& C_{0}(1+T)e^{C_{0}(1+T^{\frac{7}{4}})\Phi(\varepsilon)e^{CV_{\varepsilon}(T)}}+ C_{0}\int_{0}^{T}\big(1+(1+t^{\frac{11}{4}})e^{CV_{\varepsilon}(t)}\Phi(\varepsilon)\big)(1+V_{\varepsilon}(t))dt\\
 &\le& C_{0}e^{C_0 T}\exp\{{C_0(1+T^{\frac{15}{4}})\Phi(\varepsilon)e^{CV_{\varepsilon}(T)}}\}.
\end{eqnarray}
We choose $T_{\varepsilon}$ such that,
\begin{equation}\label{e9}
e^{e^{C_{0}T_{\varepsilon}}}=\Phi^{-\frac{1}{2}}(\varepsilon),
\end{equation}
then we claim that for small $\varepsilon$ and $0\le t\le T_{\varepsilon}$ we have
\begin{equation}\label{e10}
e^{CV_{\varepsilon}(t)}\le\Phi^{-\frac{2}{3}}(\varepsilon).
\end{equation}
Indeed, define
$$J_{{\varepsilon}}:=\big\{t\in[0,T_{\varepsilon}]; e^{CV_{\varepsilon}(t)}\le\Phi^{-\frac{2}{3}}(\varepsilon)\big\}.$$
This set is nonempty since $0\in J_{{\varepsilon}}$. It is also  closed by the continuity of the mapping $t\mapsto V_{\varepsilon}(t)$. It remains to prove that $J_{{\varepsilon}}$ is an open subset of $[0,T_\varepsilon]$ and thus we deduce $J_{{\varepsilon}}=[0,T_{\varepsilon}].$ Let $t\in J_{{\varepsilon}},$ then using \eqref{e8} we get for small $\varepsilon$
\begin{equation}\label{11}
e^{CV_{\varepsilon}(t)}\le C_0e^{\exp\{ C_0T_\varepsilon+C_{0}T_{\varepsilon}^{\frac{15}{4}}\Phi^{\frac{1}{3}}(\varepsilon)\}}\cdot
\end{equation}
From \eqref{e9}, we infer
\begin{eqnarray*}
C_0T_{\varepsilon}&=&\log\log\Phi^{-\frac{1}{2}}(\varepsilon)\approx \log\log{\Psi\big(\log(\frac{1}{\varepsilon^{\frac{1}{8}}})\big)}\cdot
\end{eqnarray*}
Since
$$
\lim_{\varepsilon\to0}\Phi^{\frac13}(\varepsilon)\big\{\log\log\Phi^{-\frac{1}{2}}(\varepsilon)\big\}^{\frac{15}{4}}=0,
$$
then  for sufficiently small $\varepsilon$ and for $t\in[0,T_{\varepsilon}],$ we get
\begin{eqnarray*}
e^{CV_{\varepsilon}(t)}&\le& 2C_0\Phi^{-\frac{1}{2}}(\varepsilon)\\
&<& \Phi^{-\frac{2}{3}}(\varepsilon).
\end{eqnarray*}
This proves that $t$ belongs to the interior of $J_{{\varepsilon}}$ and consequently $J_{{\varepsilon}}$ is an open subset of $[0,T_\varepsilon]$. Finally we get  $J_{{\varepsilon}}=[0,T_{\varepsilon}].$ From the previous estimate, we obtain for all $T\in [0,T_\varepsilon]$
\begin{equation}\label{ee8}
(1+T^{\frac{15}{4}})e^{CV_{\varepsilon}(T)}\Phi(\varepsilon)\lesssim \{\log\log\Phi^{-\frac{1}{2}}(\varepsilon)\big\}^{\frac{15}{4}}\Phi^{\frac13}(\varepsilon)\lesssim 1.
\end{equation}
Inserting \eqref{ee8} into \eqref{e8}, yields for $0\le T\le T_{\varepsilon}$
\begin{equation}\label{ee9}
V_{\varepsilon}(T)\le C_{0}e^{C_{0}T}.
\end{equation}
In particular we have obtained
$$\Vert\nabla v_{\varepsilon}\Vert_{L_{T}^{1}L^\infty}\le C_{0}e^{C_{0}T}.$$
Plugging the estimate \eqref{ee8} into \eqref{e7}, we get for small $\varepsilon$ and for $T\in[0,T_{\varepsilon}]$ that
\begin{eqnarray}\label{dig}
\nonumber \Vert(\textnormal{div}\,v_{\varepsilon},\nabla c_{\varepsilon})\Vert_{L_{T}^{1}B_{\infty,1}^{0}}&\le& C_{0}\{\log\log\Phi^{-\frac{1}{2}}(\varepsilon)\big\}^{\frac{7}{4}}\Phi^{\frac13}(\varepsilon)\\
&\le& C_{0}\Phi^{\frac14}(\varepsilon).
\end{eqnarray}
Using Proposition \ref{pr2},  \eqref{e10} and \eqref{choix1} we obtain for small $\varepsilon$,
\begin{eqnarray}\label{digg2}
\nonumber\Vert\mathbb{Q}v_{\varepsilon}\Vert_{L_{T}^{4}L^\infty}+\Vert c_{\varepsilon}\Vert_{L_{T}^{4}L^\infty}&\le& C_{0}\varepsilon^{\frac14}(1+T) e^{CV_{\varepsilon}(T)}\\
\nonumber&\le& C_{0}\varepsilon^{\frac14}\,\Phi^{-\frac23}(\varepsilon)\log\log\Phi^{-\frac{1}{2}}(\varepsilon)\\
\nonumber&\le& C_0\Phi^{\frac43}(\varepsilon)\log\log\Phi^{-\frac{1}{2}}(\varepsilon)\\
&\le& C_{0}\Phi(\varepsilon).
\end{eqnarray}
Let us now move to the estimate of the vorticity. First, recall that  $\omega_\varepsilon$ satisfies the compressible transport equation
$$
\partial_t\omega_{\varepsilon}+v_\varepsilon\cdot\nabla\omega_{\varepsilon}+\omega_{\varepsilon}\, \textnormal{div}\,v_\varepsilon=0.
$$
Consequently we get by using Gronwall's inequality that
$$
\|\omega_\varepsilon(t)\|_{L^\infty}\le\|\omega_{0,\varepsilon}\|_{L^\infty}e^{\|\textnormal{div}\,v_\varepsilon\|_{L^1_tL^\infty}}.
$$
It suffices to apply  \eqref{dig}, leading  to the following estimate: for all $ T\in[0,T_\varepsilon]$
$$
\|\omega_\varepsilon(T)\|_{L^\infty}\le C_0.
$$
Finally, to estimate $\Vert(v_{\varepsilon},c_{\varepsilon})(T)\Vert_{B_{2,1}^{2,\Psi}}$ we use Proposition \ref{f1} combined with \eqref{ee9} to find,
\begin{eqnarray}\label{croissance}
\nonumber\Vert(v_{\varepsilon},c_{\varepsilon})(T)\Vert_{B_{2,1}^{2,\Psi}}&\le& C\Vert(v_{0,\varepsilon},c_{0,\varepsilon})\Vert_{B_{2,1}^{2,\Psi}}e^{CV_{\varepsilon}(T)}\\
&\le& C_{0}e^{\exp\{C_{0}T\}}.  
\end{eqnarray}
We observe that in Theorem \ref{ttheo2} the formula for the  lifespan   involves the quantity $\log\log\Psi(\log\varepsilon^{-1})$ instead of  $\log\log\Psi(\log\varepsilon^{-\frac18})$. This is due to the fact that for any given $\lambda>0$  the {function defined by  $\Psi_\lambda(x):=\Psi(\lambda x)$} belongs to the same  class $\mathcal{U}_\infty$.
 
This completes the  proof of the first part of  Theorem \ref{ttheo2}. 
\end{proof}

\subsection{Incompressible limit}  In this paragraph we will sketch the proof  of the second part of Theorem \ref{ttheo2} which deals with the incompressible limit.
\begin{proof}
We will proceed in two steps. In the first one, we will prove that for any fixed $T>0,$ the family $(\mathbb{P}v_{\varepsilon})_{\varepsilon}$ converges strongly in $L_{T}^{\infty}L^2$,  when $\varepsilon\to 0$, to the solution  $v$ of the incompressible Euler equations with initial data $v_0$. It is worthy pointing out  that similarly to the Remark \ref{rimk1}-$(2)$ we can show  that the limit $v_0$ belongs also to the same space $B_{2,1}^{2,\Psi}.$  In the second step we will show how to get the strong convergence in the Besov spaces $B_{2,1}^{2,\widetilde\Psi}$.  We mention that the limit system \eqref{c2} is globally well-posed if $v_0\in B_{2,1}^{2,\Psi}$. Indeed, combining the embedding $B_{2,1}^{2,\Psi}\hookrightarrow B_{2,1}^2$ with the Vishik's result \cite{vis} we get that the system \eqref{c2} admits a unique global solution such that 
$v\in \mathcal{C}(\RR_+; B_{2,1}^2),$ with the following Lipschitz bound 
$$
\|\nabla v(t)\|_{L^\infty}\le C_0e^{C_0 t}.
$$
This latter a priori estimate together with Proposition \ref{f1}, which remains valid for \eqref{c2}, gives the global persistence of  the initial regularity $B_{2,1}^{2,\Psi}.$ More precisely, we get
\begin{equation}\label{log-vis}
\|v(t)\|_{B_{2,1}^{2,\Psi}}\le C_0e^{\exp{C_0 t}}.
\end{equation}
Let $\varepsilon>0$ and set 
$$
\zeta_{\varepsilon}:=v_{\varepsilon}-v,\quad w_{\varepsilon}=\mathbb{P}v_{\varepsilon}-\, v.
$$ 
By applying Leray's projector $\mathbb{P}$ to the first equation of \eqref{C1}, we obtain
$$
\partial_{t}\mathbb{P}v_{\varepsilon}+\mathbb{P}(v_{\varepsilon}\cdot\nabla v_{\varepsilon})=0.
$$
Taking the difference between the above equation and \eqref{c2} yields to
$$
\partial_{t}\,w_{\varepsilon}+\mathbb{P}(v_{\varepsilon}\cdot\nabla\zeta_{\varepsilon})+\mathbb{P}(\zeta_{\varepsilon}\cdot\nabla v)=0.
$$
We take the $L^2$ inner product of the above equation with $\,w_{\varepsilon},$ integrating by parts and using the identities 
$$
\zeta_{\varepsilon}=\,w_{\varepsilon}+\mathbb{Q}v_{\varepsilon},\quad \mathbb{P}\,w_{\varepsilon}=\,w_{\varepsilon},
$$ 
we get,
\begin{eqnarray*}
\frac{1}{2}\frac{d}{dt}\Vert\,w_{\varepsilon}(t)\Vert^{2}_{L^2}&=& -\int_{\RR^2}(v_{\varepsilon}\cdot\nabla\zeta_{\varepsilon})\mathbb{P}\,w_{\varepsilon}dx-\int_{\RR^2}(\zeta_{\varepsilon}\cdot\nabla v)\mathbb{P}\,w_{\varepsilon}dx\\
&=& -\int_{\RR^2}\big(v_{\varepsilon}\cdot\nabla(\,w_{\varepsilon}+\mathbb{Q}v_{\varepsilon})\big)\,w_{\varepsilon}dx-\int_{\RR^2}\big((\,w_{\varepsilon}+\mathbb{Q}v_{\varepsilon})\cdot\nabla v\big)\,w_{\varepsilon}dx\\
&=& \frac{1}{2}\int_{\RR^2}\vert\,w_{\varepsilon}\vert^{2}\textnormal{div}\,v_{\varepsilon}dx-\int_{\RR^2}(\,w_{\varepsilon}\cdot\nabla v)\,w_{\varepsilon}dx\\
&-& \int_{\RR^2}(v_{\varepsilon}\cdot\nabla\mathbb{Q}v_{\varepsilon}+\mathbb{Q}v_{\varepsilon}\cdot\nabla{ v})\,w_{\varepsilon}dx\\
&\le& \Big(\frac{1}{2}\Vert\textnormal{div}\,v_{\varepsilon}(t)\Vert_{L^\infty}+\Vert\nabla v(t)\Vert_{L^\infty}\Big)\Vert\,w_{\varepsilon}(t)\Vert^{2}_{L^2}\\
&+& \Big(\Vert v_{\varepsilon}(t)\Vert_{L^2}\Vert\nabla\mathbb{Q}v_{\varepsilon}(t)\Vert_{L^\infty}+\Vert\mathbb{Q}v_{\varepsilon}(t)\Vert_{L^\infty}\Vert\nabla v(t)\Vert_{L^2}\Big)\Vert\,w_{\varepsilon}(t)\Vert_{L^2}.
\end{eqnarray*}
This implies that
\begin{eqnarray*}
\frac{d}{dt}\Vert\,w_{\varepsilon}(t)\Vert_{L^2}&\le& \big(\frac{1}{2}\Vert\textnormal{div}\,v_{\varepsilon}(t)\Vert_{L^\infty}+\Vert\nabla v(t)\Vert_{L^\infty}\big)\Vert\,w_{\varepsilon}(t)\Vert_{L^2}\\
&+& \Vert v_{\varepsilon}(t)\Vert_{L^2}\Vert\nabla\mathbb{Q} v_{\varepsilon}(t)\Vert_{L^\infty}+\Vert\mathbb{Q}v_{\varepsilon}(t)\Vert_{L^\infty}\Vert\nabla v(t)\Vert_{L^2}.
\end{eqnarray*}
Integrating in time and using Gronwall's inequality yield to
\begin{equation}\label{x4}
\Vert\,w_{\varepsilon}(t)\Vert_{L^2}\le\big(\Vert w^{0}_{\varepsilon}\Vert_{L^2}+F_{\varepsilon}(t)\big)e^{\frac{1}{2}\Vert\textnormal{div}\,v_{\varepsilon}\Vert_{L_{t}^{1}L^\infty}+\Vert\nabla{ v}\Vert_{L_{t}^{1}L^\infty}},
\end{equation}
where
\begin{eqnarray*}
F_{\varepsilon}(t)&:=&\Vert v_{\varepsilon}\Vert_{L_{t}^{\infty}L^2}\Vert\nabla\mathbb{Q}v_{\varepsilon}(t)\Vert_{L_{t}^{1}L^\infty}+\Vert\mathbb{Q}v_{\varepsilon}\Vert_{L_{t}^{1}L^\infty}\Vert\nabla v\Vert_{L_{t}^{\infty}L^2}\\
&\lesssim& \Vert v_{\varepsilon}\Vert_{L_{t}^{\infty}L^2}\Vert\nabla\mathbb{Q}v_{\varepsilon}\Vert_{L_{t}^{1}L^\infty}+\Vert\mathbb{Q}v_{\varepsilon}\Vert_{L_{t}^{1}L^\infty}\Vert \omega\Vert_{L_{t}^{\infty}L^2},
\end{eqnarray*}
where $\omega$ denotes the vorticity of $v$.To estimate $\Vert\nabla\mathbb{Q} v_{\varepsilon}\Vert_{L^\infty}$ we use Bernstein inequality, 
\begin{eqnarray*}
\Vert\nabla\mathbb{Q}v_{\varepsilon}\Vert_{L^\infty}&\le& \Vert\Delta_{-1}\nabla\mathbb{Q}v_{\varepsilon}\Vert_{L^\infty}+\sum_{q\ge 0}\Vert\Delta_{q}\nabla\mathbb{Q}v_{\varepsilon}\Vert_{L^\infty}\\
&\lesssim& \Vert\mathbb{Q}v_{\varepsilon}\Vert_{L^\infty}+\sum_{q\ge 0}\Vert\Delta_{q}\textnormal{div}\,v_{\varepsilon}\Vert_{L^\infty}\\
&\lesssim& \Vert\mathbb{Q}v_{\varepsilon}\Vert_{L^\infty}+\Vert\textnormal{div}\,v_{\varepsilon}\Vert_{B_{\infty,1}^{0}}.
\end{eqnarray*}
Integrating in time
$$\Vert\nabla\mathbb{Q}v_{\varepsilon}\Vert_{L_{t}^{1}L^\infty}\le\Vert\mathbb{Q}v_{\varepsilon}\Vert_{L_{t}^{1}L^\infty}+\Vert\textnormal{div}\,v_{\varepsilon}\Vert_{L_{t}^{1}B_{\infty,1}^{0}}.$$
By virtue of \eqref{dig} and \eqref{digg2} one has for small $\varepsilon$ and $t\in [0,T],$
$$
\Vert\nabla\mathbb{Q}v_{\varepsilon}\Vert_{L_{t}^{1}L^\infty}+\Vert\mathbb{Q}v_{\varepsilon}\Vert_{L_{t}^{1}L^\infty}\le C_0 \Phi^{\frac14}(\varepsilon).
$$
On the other hand we get from Proposition \ref{f1}-(1) and \eqref{dig}
$$\Vert v_{\varepsilon}(t)\Vert_{L_{t}^{\infty}L^2}\le C_0.$$
To estimate $\|\omega(t)\|_{L^2}$ we use the fact that the vorticity is transported for the incompressible flow and thus
$$\|\omega(t)\|_{L^2}=\|\omega_0\|_{L^2}.$$
Therefore we obtain for $t\in[0,T]$ and for small $\varepsilon,$
$$F_{\varepsilon}(t)\le C_0 \Phi^{\frac14}(\varepsilon).$$
According to Vishik's result \cite{vis} we have for all $t\geq 0$
$$\|\nabla v(t)\|_{L^\infty}\leq C_0 e^{C_0 t}.$$
Inserting the previous estimates in \eqref{x4} we find for all $t\in[0,T]$ and small values of $\varepsilon,$
\begin{equation}\label{taux de conv}
\|w_\varepsilon(t)\|_{L^2}\le C_0e^{\exp{C_0 t}}\big(\|\mathbb{P}v_{0,\varepsilon}-v_0\|_{L^2}+ \Phi^{\frac14}(\varepsilon)\big).
\end{equation}
This achieves the proof of the strong convergence in $L^\infty_{loc}(\RR_+;L^2)$.

Let us now turn to the proof of  the strong convergence in the space $L^\infty_TB_{2,1}^{2,\widetilde\Psi}$. We recall that $\widetilde\Psi$ is any element  of  the set $\mathcal{U}, $ see Definition \ref{heter}, and satisfying in addition  the assumption
$$\lim_{q\to +\infty}\frac{\widetilde\Psi(q)}{\Psi(q)}=0.$$
By using \eqref{croissance} and the continuity of Leray's projector on the spaces $B_{2,1}^{2,\Psi}$ we get  for $t\in[0,T]$ and for small $\varepsilon$
\begin{equation}\label{log-viss}
\|\mathbb{P}v_\varepsilon(t)\|_{B_{2,1}^{2,\Psi}}\le C_0e^{\exp{C_0 t}}.
\end{equation}
Now let $N\in \NN^\star$ be an arbitrary number. Then by \eqref{taux de conv}, \eqref{log-vis} and \eqref{log-viss} we obtain
\begin{eqnarray*}
\|w_\varepsilon(t)\|_{B_{2,1}^{2,\widetilde\Psi}}&=&\sum_{q=-1}^{N}\widetilde\Psi(q)2^{2q}\|\Delta_qw_\varepsilon(t)\|_{L^2}+\sum_{q>N}\frac{\widetilde\Psi(q)}{\Psi(q)}\Psi(q)2^{2q}\|\Delta_qw_\varepsilon(t)\|_{L^2}\\
&\le& C_0e^{\exp{C_0 t}}\big(\|\mathbb{P}v_{0,\varepsilon}-v_0\|_{L^2}+ \Phi^{\frac14}(\varepsilon)\big)\widetilde\Psi(N)2^{2N}+\rho_N \|(\mathbb{P}v_\varepsilon,v)(t)\|_{B_{2,1}^{2,\Psi}}\\
&\le& C_0e^{\exp{C_0 t}}\Big(\big(\|\mathbb{P}v_{0,\varepsilon}-v_0\|_{L^2}+ \Phi^{\frac14}(\varepsilon)\big)\widetilde\Psi(N)2^{2N}+\rho_N\Big),
\end{eqnarray*}
with
 $$\rho_N:=\max_{q\geq N}\frac{\widetilde\Psi(q)}{\Psi(q)}.$$
 We remark that the $(\rho_N)_N$ is a decreasing sequence to zero. Consequently, 
 $$ \limsup_{\varepsilon\to 0^+} \|w_\varepsilon\|_{L^\infty_TB_{2,1}^{2,\widetilde\Psi}}\le C_0e^{\exp{C_0 T}}\rho_N. $$
 Combining this estimate with $\displaystyle{\lim_{N\to +\infty}\rho_N=0},$ we get
 $$\lim_{\varepsilon\to 0^+} \|w_\varepsilon\|_{L^\infty_TB_{2,1}^{2,\widetilde\Psi}}=0.$$
 This completes the proof of the strong convergence.
\end{proof}

\end{document}